\documentclass{IEEEtran}

\usepackage{cite}
\usepackage{amsmath,amssymb,amsfonts,amsthm,mathtools}
\usepackage[hidelinks]{hyperref}
\usepackage[nameinlink,noabbrev]{cleveref}

\def\BibTeX{{\rm B\kern-.05em{\sc i\kern-.025em b}\kern-.08em
    T\kern-.1667em\lower.7ex\hbox{E}\kern-.125emX}}

\newtheorem{theorem}{Theorem}
\newtheorem{corollary}[theorem]{Corollary}
\newtheorem{proposition}[theorem]{Proposition}
\newtheorem{lemma}[theorem]{Lemma}
\theoremstyle{remark}

\newcommand{\R}{\mathbb{R}}
\newcommand{\Q}{\mathbb{Q}}
\newcommand{\B}{\mathbb{B}}
\newcommand{\sat}{\operatorname{sat}}

\title{The Global Asymptotic Stability Problem for Linear MPC Is Undecidable}
\author{Johan Löfberg%
\thanks{This work was supported by the Swedish Government Agency VINNOVA
through the SEDDIT competence center program.}%
\thanks{Johan Löfberg is with the Division of Automatic Control, Department
of Electrical Engineering, Linköping University, SE-581 83 Linköping,
Sweden (e-mail: johan.lofberg@liu.se).}}

\begin{document}
\maketitle

\begin{abstract}
We prove that deciding global asymptotic stability for constrained finite-horizon linear model predictive control is undecidable. This holds at horizon one with identity state, input, and terminal weights, unique optimizers, and global feasibility. Separate reductions cover predicted-state boxes, hard input boxes, and quadratically softened input boxes. A fourth reduction fixes the state and input dimensions to three and six. Hence undecidability is not caused by long horizons, growing dimensions, failures of recursive feasibility, or nonuniqueness.
\end{abstract}


\section{Introduction}

Model predictive control is among the most successful modern control methodologies, with significant use in industrial process control and an expanding range of other applications \cite{GarciaEtAl1989,QinBadgwell2003,RawlingsMayneDiehl2024}.  Its appeal lies in its ability to handle multivariable dynamics and constraints directly within an online optimization problem.  A somewhat disturbing fact, however, is that closed-loop stability is not intrinsic to the basic finite-horizon MPC formulation.  It must either be imposed through deliberate controller tuning and additional design ingredients, such as terminal constraints, terminal costs, or suitably chosen local control laws \cite{MayneEtAl2000,ScokaertEtAl1999}, guaranteed without terminal ingredients by choosing a sufficiently long prediction horizon under appropriate controllability conditions \cite{Gruene2012}, or established a posteriori for the resulting receding-horizon law.  Such analysis may work directly with the optimization-based controller \cite{Primbs2001,KordaJones2017,Lofberg2012,SimonLofberg2016}, or first derive its explicit piecewise-affine representation \cite{BemporadEtAl2002} and then search for Lyapunov or invariant-set certificates for the resulting closed-loop dynamics \cite{FerrariTrecateEtAl2002,HovdOlaru2010,RubagottiEtAl2011,GroffEtAl2023}.  Since the controller is implemented in feedback, with the optimization resolved at every sampling instant, this is a question about the global dynamics generated jointly by the plant, feasible set, and active constraints.  In particular, do all feasible trajectories remain admissible and converge to the desired equilibrium?

More precisely, consider regulation of the discrete-time linear plant
\[
  x_{k+1}=F x_k+G u_k.
\]
The variables are expressed in deviation coordinates about the controlled equilibrium $(x_k,u_k)=(0,0)$, so the control objective is convergence to the origin.
Given a horizon $N$, let $\mathbf{x}=(x_{0|k},\ldots,x_{N|k})$ and
$\mathbf{u}=(u_{0|k},\ldots,u_{N-1|k})$ denote the predicted state and input
sequences.  The standard quadratic MPC controller considered here solves
\begin{subequations}
\label{eq:standard-mpc}
\begin{align}
\min_{\mathbf{x},\mathbf{u}}\quad&
\sum_{i=0}^{N-1}
\left(x_{i|k}^{\mathsf T}Qx_{i|k}
+u_{i|k}^{\mathsf T}Ru_{i|k}\right)
+x_{N|k}^{\mathsf T}Px_{N|k},\\
\text{subject to}\quad&
 x_{0|k}=x_k,\\
&x_{i+1|k}=Fx_{i|k}+Gu_{i|k},
\qquad i=0,\ldots,N-1,\\
&(x_{i|k},u_{i|k})\in\mathcal Z,
\qquad i=0,\ldots,N-1,\\
&x_{N|k}\in\mathcal X_f,
\end{align}
\end{subequations}
where $Q\succeq0$, $R\succ0$, and $P\succeq0$, while $\mathcal Z$ and $\mathcal X_f$ are polyhedral sets containing $(0,0)$ and $0$, respectively; selected state, input, or terminal constraints may be absent.  The receding-horizon law is $\kappa_N(x_k)=u_{0|k}^{\star}(x_k)$, and the implemented closed loop is
\[
  x_{k+1}=Fx_k+G\kappa_N(x_k).
\]
We denote by $\mathcal X_N$ the set of states for which \eqref{eq:standard-mpc} is feasible.  Softened MPC problems, obtained by adding nonnegative slack variables to selected inequalities and penalizing them in the objective, are also considered below.  For such an instance, the rational matrices defining the slack incidence and the rational quadratic slack penalty are included in its finite encoding.

Since $R\succ0$ and the dynamics determine the states from the inputs, the
hard-constrained optimizer is unique; the positive-definite slack penalty gives
the same property in the softened construction.

For linear MPC with quadratic cost and polyhedral constraints, the receding-horizon law is continuous piecewise affine under the assumptions used here \cite{BemporadEtAl2002}.  This makes local analysis and region-wise calculation natural, but it does not by itself settle global behavior across switching surfaces.  We therefore consider the following decision problem.  An MPC instance is specified by a finite horizon and finitely many rational matrices and vectors defining the plant, quadratic cost, polyhedral inequalities, and, when present, slack variables and their penalties; the state and input dimensions are part of the instance and may vary.  Given such an instance, decide whether the origin is asymptotically stable relative to the feasible set $\mathcal X_N$.  Here, asymptotic stability relative to $\mathcal X_N$ means that every closed-loop trajectory initialized in $\mathcal X_N$ remains in $\mathcal X_N$ and converges to the origin, and that the origin is Lyapunov stable relative to $\mathcal X_N$.

The main result of this paper rigorously confirms the negative conclusion intuitively suggested by the broader piecewise-affine literature: the global asymptotic-stability decision problem for constrained finite-horizon linear MPC is undecidable.  In fact, undecidability persists even when the horizon is fixed to $N=1$ and the quadratic weights to $Q=R=P=I$.  In a separate fixed-dimensional reduction, both the state and input dimensions are constants, namely $n_x=3$ and $n_u=6$; the finite description of the encoded computation then appears only in the number and coefficients of the mixed polyhedral inequalities.

\section{A known undecidable stability problem}

Undecidability of stability is already known for general piecewise-affine dynamics; see \cite{BlondelTsitsiklis2000} for a broader survey of complexity results in systems and control.  Blondel, Bournez, Koiran, Papadimitriou, and Tsitsiklis proved undecidability for rational piecewise-affine discrete-time systems in a class that permits discontinuous maps, already in dimension two \cite{BlondelEtAlPWA2001}.  More strongly, Blondel, Bournez, Koiran, and Tsitsiklis proved undecidability for continuous piecewise-affine systems in dimension three and for the narrower class of saturated linear systems used below \cite{BlondelEtAl2001}.  Thus continuity alone does not restore decidability.

Although these results make undecidability for linear MPC highly plausible and expected, they do not establish it.  MPC closed-loop maps form a very particular structured subclass of continuous piecewise-affine maps: the feedback must be the optimizer of a finite-horizon quadratic program whose parameter dependence and objective are tied to the prediction model.  An undecidable problem on a larger class may become decidable after restriction to such a subclass.  Equivalently, the familiar implication $\text{linear MPC}\Longrightarrow\text{piecewise-affine closed loop}$ has the wrong direction for transferring undecidability.  What is needed is a stability-preserving construction in the reverse direction, from a known undecidable \emph{source problem} to MPC.

A related reduction has been given for linear plants controlled by rational ReLU networks: Korda observed that deciding asymptotic stability remains undecidable when saturation is represented by a small ReLU network \cite{Korda2022}.  Korda's construction does not directly yield the MPC results established here: it represents saturation by a ReLU network rather than as the optimizer of a plant-consistent finite-horizon quadratic program.  The contribution here is a family of explicit reductions that realize the undecidable dynamics within distinct MPC subclasses, allowing predicted-state, hard-input, and soft-input formulations to be studied separately.

The source problem for our reduction is primarily global asymptotic stability of saturated linear systems studied by Blondel et al.  Let $n$ be a positive integer and let $A\in\Q^{n\times n}$.  The state is $x_k\in\R^n$.  For any vector $v\in\R^n$, define componentwise saturation by
\[
  \sat(v)_i =
  \begin{cases}
    -1, & v_i<-1,\\
    v_i, & -1\le v_i\le 1,\\
    1, & v_i>1.
  \end{cases}
\]
The associated saturated linear system is
\begin{equation}
    x_{k+1}=\sat(Ax_k),
    \qquad x_0\in\R^n.
    \label{eq:saturated-system}
\end{equation}

\begin{theorem}[Saturated linear system source \cite{BlondelEtAl2001}]
\label{thm:saturated-source}
Given a rational matrix $A$, it is undecidable whether the origin of \eqref{eq:saturated-system} is globally asymptotically stable on $\R^n$.
\end{theorem}

We will also use the following origin-fixing form of
\cite[Theorem~7.1]{BlondelEtAl2001}.

\begin{theorem}[Fixed-dimensional continuous PWA source]
\label{thm:fixed-pwa-source}
Given a rational continuous piecewise-affine map
$\phi:\R^3\to\R^3$ satisfying $\phi(0)=0$, it is undecidable whether the
origin of $x_{k+1}=\phi(x_k)$ is globally asymptotically stable.
\end{theorem}

We encode a rational PWA map by a finite rational polyhedral partition and
rational affine pieces.  The hard instances in
\cite[Theorem~7.1]{BlondelEtAl2001} have such descriptions.

Since $\phi(0)=0$ is necessary for asymptotic stability of the origin and is
exactly checkable for a rational PWA map, the source problem may without loss
of generality be restricted to maps fixing the origin; all other maps are
immediate no-instances.

\section{The MPC stability decision problem is undecidable}

We now construct MPC problems whose receding-horizon closed loops reproduce the saturated system \eqref{eq:saturated-system} exactly, either directly or as an autonomous component.  Each construction maps a rational matrix $A$ to a rational horizon-one MPC instance using only fixed rational operations.  Hence the map is effective (indeed, polynomial-time) under standard binary encoding of rational data. To isolate what does and does not contribute to undecidability, we derive three alternative reductions to instances of the MPC stability decision problem. One uses predicted-state constraints only, one uses input constraints only, and one uses quadratically softened input constraints. Together, the constructions show that the obstruction is attributable neither specifically to input constraints nor to state or output constraints. Every constructed MPC problem is feasible for every state, so recursive feasibility is automatic in these reductions and is not the mechanism encoding undecidability. In every construction the optimizer at the origin is zero, so the origin is a closed-loop equilibrium.

\begin{theorem}[Undecidability of the MPC stability decision problem]
\label{thm:mpc-undecidability}
The global asymptotic-stability decision problem for rational constrained finite-horizon linear MPC is undecidable.  Equivalently, no algorithm can terminate on every instance and correctly decide whether the resulting closed loop is globally asymptotically stable.  This remains true even when the horizon is fixed to one and the state, terminal, and input weights are identity matrices of the appropriate dimensions, already for each of the following subclasses: a one-step predicted-state/terminal box with no input constraints; an input box with no state or terminal constraints; and a quadratically softened input box with no hard bounds on the physical state or input and with a fixed rational slack penalty.  In every case the quadratic program is strictly convex, has a unique optimizer, and is feasible for every state.
\end{theorem}

\begin{proof}
The propositions in the following subsections give, for every rational matrix $A$, a rational MPC instance $\mathcal I(A)$ whose relevant closed-loop component is exactly $x_{k+1}=\sat(Ax_k)$. They establish that $\sat(A\,\cdot)$ is globally asymptotically stable if and only if $\mathcal I(A)$ is globally asymptotically stable.
Thus $A\mapsto\mathcal I(A)$ is a computable many-one reduction from saturated-system stability to MPC stability.  An algorithm deciding global asymptotic stability for general linear MPC would therefore decide the same property for rational saturated linear systems, contradicting \cref{thm:saturated-source}.
\end{proof}

The remainder of this section establishes the three exact realizations used in the theorem.

\subsection{Predicted-state constraints only}
\label{sec:hard-state}

This is the most direct realization.  Given $A\in\Q^{n\times n}$, define the linear system
\begin{equation}
  x_{k+1}=2Ax_k+u_k,
  \label{eq:state-only-plant}
\end{equation}
with $u_k\in\R^n$.  Denote the unit box by $\B_n\coloneqq[-1,1]^n$.  At every state $x_k\in\R^n$, solve
\begin{subequations}
\label{eq:state-only-mpc}
\begin{align}
  \min_{u_k,x_{k+1}}\quad & \|x_k\|^2+\|u_k\|^2+\|x_{k+1}\|^2,\\
  \text{subject to}\quad & x_{k+1}=2Ax_k+u_k,\\
  & x_{k+1}\in\B_n.
\end{align}
\end{subequations}
There are no input constraints and the only
inequalities are a hard box constraint on the one-step predicted state, which can
equivalently be viewed as a terminal or predicted-output constraint.

\begin{proposition}[Exact projection realization]
\label{prop:state-only}
Problem \eqref{eq:state-only-mpc} is feasible for every $x_k\in\R^n$, has a unique
optimizer, and generates the closed-loop map
\[
  x_{k+1}=\sat(Ax_k).
\]
\end{proposition}

\begin{proof}
Put $z=x_{k+1}$ and eliminate $u_k=z-2Ax_k$.  The decision-dependent objective is
\begin{align*}
  \|z-2Ax_k\|^2+\|z\|^2
  &=2\|z-Ax_k\|^2+2\|Ax_k\|^2.
\end{align*}
Thus $z$ is the Euclidean projection of $Ax_k$ onto $\B_n$, so
$z^\star=\sat(Ax_k)$.  Feasibility is global because, for example, $z=0$ and
$u_k=-2Ax_k$ are feasible for every $x_k$.  Strict convexity gives uniqueness.
\end{proof}

Together with \cref{thm:saturated-source}, \cref{prop:state-only} already proves that
the stability decision problem is undecidable for horizon-one MPC with only
predicted-state constraints; the feasible set in this construction is all of
$\R^n$. The following corollary strengthens this conclusion by showing that
undecidability persists when the feasible set is compact and invariant.

\begin{corollary}[Compact invariant feasible set]
\label{cor:compact-box}
Add $x_{0|k}=x_k\in\B_n$ to \eqref{eq:state-only-mpc}, thereby restricting
the current-state parameter.  The resulting MPC feasible set is exactly the
compact invariant polytope $\B_n$, and deciding asymptotic stability relative
to this feasible set is undecidable.
\end{corollary}

\begin{proof}
The optimizer remains $x_{k+1}=\sat(Ax_k)$ for every $x_k\in\B_n$, and hence $x_{k+1}\in\B_n$.  Moreover, every trajectory of the source system \eqref{eq:saturated-system} enters $\B_n$ after one step.  Consequently, the source system is globally asymptotically stable on $\R^n$ if and only if it is asymptotically stable relative to $\B_n$: global convergence follows after the first step, while Lyapunov stability is a local property and $\B_n$ contains a neighborhood of the origin.  The claim now follows from \cref{thm:saturated-source}.
\end{proof}

\subsection{Input constraints only}
\label{sec:input-only}

The next two constructions augment the source state with a component that has
no memory of its own.  We first introduce a minor stability equivalence used for both.

\begin{lemma}[Stability under memoryless augmentation]
\label{lem:static-augmentation}
Let $f:\R^n\to\R^n$ and $h:\R^n\to\R^m$ be continuous, with $f(0)=0$ and $h(0)=0$.  The origin of
\[
  (z_{k+1},y_{k+1})=(f(z_k),h(z_k))
\]
is globally asymptotically stable if and only if the origin of $z_{k+1}=f(z_k)$ is globally asymptotically stable.
\end{lemma}

\begin{proof}
Suppose the augmented system is globally asymptotically stable.  For any
$z_0$, initialize it at $(z_0,0)$.  Its first component then satisfies
$z_k=f^k(z_0)$.  Convergence of the augmented trajectory therefore implies
$f^k(z_0)\to0$.  Moreover, since $\|(z_0,0)\|=\|z_0\|$ and
$\|z_k\|\le\|(z_k,y_k)\|$, Lyapunov stability of the augmented system implies
Lyapunov stability of $z_{k+1}=f(z_k)$.

Conversely, suppose that $z_{k+1}=f(z_k)$ is globally asymptotically stable.  For every $k\ge1$,
\[
  (z_k,y_k)=\bigl(f^k(z_0),h(f^{k-1}(z_0))\bigr).
\]
Global convergence follows from continuity of $h$ at the origin.  For Lyapunov stability, fix $\varepsilon>0$ and choose $\rho>0$ such that $\|\xi\|<\rho$ and $\|\zeta\|<\rho$ imply $\|(\xi,\zeta)\|<\varepsilon$.  By continuity of $h$, choose $\eta>0$ such that $\|z\|<\eta$ implies $\|h(z)\|<\rho$.  Lyapunov stability of $f$ then gives a $\delta_z>0$ such that $\|z_0\|<\delta_z$ implies $\|f^k(z_0)\|<\min\{\eta,\rho\}$ for every $k\ge0$.  Choose $\delta\le\min\{\delta_z,\varepsilon\}$.  If $\|(z_0,y_0)\|<\delta$, then the initial augmented state is within $\varepsilon$, while the displayed iterate formula and the preceding bounds control every $k\ge1$.
\end{proof}

We now remove all state and output constraints.  Introduce the augmented state
\[
  x_k=\begin{bmatrix}z_k\\y_k\end{bmatrix}\in\R^{2n}
\]
and the plant
\begin{subequations}
\label{eq:input-only-plant}
\begin{align}
  z_{k+1}&=u_k,\\
  y_{k+1}&=3Az_k-u_k.
\end{align}
\end{subequations}
Equivalently,
\[
  x_{k+1}
  =
  \begin{bmatrix}0&0\\3A&0\end{bmatrix}x_k
  +
  \begin{bmatrix}I\\-I\end{bmatrix}u_k.
\]
The horizon-one controller is
\begin{subequations}
\label{eq:input-only-mpc}
\begin{align}
  \min_{u_k,x_{k+1}}\quad &\|x_k\|^2+\|u_k\|^2+\|x_{k+1}\|^2,\\
  \text{subject to}\quad &\eqref{eq:input-only-plant},\\
  &u_k\in\B_n.
\end{align}
\end{subequations}
The box on $u_k$ is the only inequality constraint.

\begin{proposition}[Exact realization using input saturation]
\label{prop:input-only}
The optimizer of \eqref{eq:input-only-mpc} is
\[
  u_k^\star=\sat(Az_k),
\]
and the MPC closed loop is
\[
  x_{k+1}=\begin{bmatrix}z_{k+1}\\y_{k+1}\end{bmatrix}
  =
  \begin{bmatrix}
    \sat(Az_k)\\
    3Az_k-\sat(Az_k)
  \end{bmatrix}.
\]
It is globally asymptotically stable if and only if
$z_{k+1}=\sat(Az_k)$ is globally asymptotically stable.
\end{proposition}

\begin{proof}
Let $v=Az_k$.  Omitting the current-state term, the objective is
\begin{align*}
  &\|u_k\|^2+\|z_{k+1}\|^2+\|y_{k+1}\|^2\\
  &\quad=\|u_k\|^2+\|u_k\|^2+\|3v-u_k\|^2\\
  &=3\|u_k-v\|^2+6\|v\|^2.
\end{align*}
Projection onto $\B_n$ gives $u_k^\star=\sat(v)$.  In particular, the
optimizer depends only on $z_k$, not on the memoryless component $y_k$.
The displayed closed-loop map follows from the plant, and the stability
equivalence follows from \cref{lem:static-augmentation} with
$f(z)=\sat(Az)$ and $h(z)=3Az-\sat(Az)$.
\end{proof}

\subsection{Quadratically softened input constraints}
\label{sec:soft-input}

The final construction removes all hard bounds on both states and inputs.  Instead, the input box is softened by an unbounded nonnegative slack.  Introduce the augmented state
\[
  x_k=\begin{bmatrix}z_k\\y_k\end{bmatrix}\in\R^{2n},
  \qquad u_k\in\R^n,
\]
and the plant
\begin{subequations}
\label{eq:soft-plant}
\begin{align}
  z_{k+1}&=-Az_k+2u_k,\\
  y_{k+1}&=4Az_k-u_k.
\end{align}
\end{subequations}
At each state $x_k$ solve
\begin{subequations}
\label{eq:soft-mpc}
\begin{align}
  \min_{u_k,s_k,x_{k+1}}\quad
  &\|x_k\|^2+\|u_k\|^2+\|x_{k+1}\|^2+6\|s_k\|^2,\\
  \text{subject to}\quad
  &\eqref{eq:soft-plant},\\
  &-\mathbf 1-s_k\le u_k\le \mathbf 1+s_k,\\
  &s_k\ge0.
\end{align}
\end{subequations}
Since the slack is unbounded above, the problem is feasible for every current state.  The displayed inequalities remain hard constraints of the quadratic program, but they impose no hard bound on the physical state or input.

\begin{lemma}[Leaky projection induced by a quadratic slack]
\label{lem:leaky}
For $v\in\R^n$, the unique optimizer of
\begin{equation}
\label{eq:leaky-problem}
\begin{aligned}
  \min_{u,s}\quad &\|u-v\|^2+\|s\|^2\\
  \text{subject to}\quad
  &-\mathbf 1-s\le u\le\mathbf 1+s,\quad s\ge0
\end{aligned}
\end{equation}
has components
\[
  u^\star=\frac12\bigl(v+\sat(v)\bigr),
  \qquad
  s_i^\star=\frac12\bigl(|v_i|-1\bigr)_+.
\]
\end{lemma}

\begin{proof}
The problem separates by coordinate.  For fixed $u_i$, the optimal slack is
$s_i=(|u_i|-1)_+$.  Eliminating it gives the continuously differentiable,
strictly convex scalar problem
\[
  \min_{u_i\in\R}
  (u_i-v_i)^2+\bigl(|u_i|-1\bigr)_+^2.
\]
If $|v_i|\le1$, the unique minimizer is $u_i=v_i$.  If $v_i>1$, stationarity above one gives $u_i=(v_i+1)/2$.  The case $v_i<-1$ is symmetric and gives $u_i=(v_i-1)/2$.  These cases are exactly $u_i^\star=(v_i+\sat(v_i))/2$.  Substitution into $s_i=(|u_i|-1)_+$ gives the stated $s_i^\star$.
\end{proof}

\begin{proposition}[Soft input constraints recover exact saturation]
\label{prop:soft}
The optimizer of \eqref{eq:soft-mpc} satisfies
\[
  u_k^\star=\frac12\bigl(Az_k+\sat(Az_k)\bigr),
\]
and the first closed-loop component obeys
\[
  z_{k+1}=\sat(Az_k).
\]
The complete augmented closed loop is globally asymptotically stable if and only if the source saturated system is globally asymptotically stable.
\end{proposition}

\begin{proof}
Set $v=Az_k$.  The part of the objective depending on $u_k$ and $s_k$ is
\begin{align*}
 &\|u_k\|^2+\|-v+2u_k\|^2+\|4v-u_k\|^2+6\|s_k\|^2\\
 &\qquad=6\|u_k-v\|^2+11\|v\|^2+6\|s_k\|^2.
\end{align*}
After division by six and removal of the constant term, the optimization over $u_k$ and $s_k$ is exactly \eqref{eq:leaky-problem}.  Hence, by \cref{lem:leaky},
\[
  u_k^\star=\tfrac12\bigl(v+\sat(v)\bigr).
\]
The first plant equation cancels the linear leakage:
\[
  z_{k+1}=-v+2u_k^\star=\sat(v).
\]
The remaining successor component is
\[
  y_{k+1}=\frac72v-\frac12\sat(v),
\]
a continuous function of $z_k$ that vanishes at the origin.  The stability equivalence follows from \cref{lem:static-augmentation} with $f(z)=\sat(Az)$ and $h(z)=\tfrac72Az-\tfrac12\sat(Az)$.
\end{proof}

\section{Undecidability with fixed state and input dimensions}
\label{sec:fixed-dimension}

The saturated-system reductions above use a state dimension inherited from the
source matrix.  We now replace that source problem by
\cref{thm:fixed-pwa-source} and use an epigraph lifting of continuous PWA maps.
The lifting is based on the constructive difference-of-convex decomposition of
continuous PWA functions \cite{KripfganzSchulze1987,HempelEtAl2013}.

Inverse parametric optimization already represents continuous PWA maps through
suitably constructed LPs or QPs \cite{HempelEtAl2013,HempelEtAl2015}.  The
contribution here is not the DC lifting itself, but showing that inflated
components make the epigraph corner optimal for the prescribed,
plant-consistent identity-weight MPC quadratic.

\begin{lemma}[Nonnegative epigraph decomposition]
\label{lem:epigraph-decomposition}
Let $\phi:\R^d\to\R^d$ be a rational continuous PWA map with $\phi(0)=0$.
There exist effectively computable rational convex PWA functions
$p_i,q_i:\R^d\to\R$, $i=1,\ldots,d$, such that
\[
  \phi_i(x)=p_i(x)-q_i(x),\qquad
  p_i(0)=q_i(0)=0,
\]
and
\[
  p_i(x)\ge |\phi_i(x)|,
  \qquad q_i(x)\ge |\phi_i(x)|
  \qquad\text{for every }x\in\R^d.
\]
Each $p_i$ and $q_i$ can be represented as the maximum of finitely many
rational affine functions.
\end{lemma}

\begin{proof}
By the constructive PWA difference-of-convex decomposition
\cite[Lemma~4]{HempelEtAl2013}, write
\[
  \phi_i(x)=a_i(x)-b_i(x),
\]
where $a_i$ and $b_i$ are rational convex PWA functions.  Since
$\phi_i(0)=0$, their values at the origin agree; denote the common value by
$c_i$.  Write $a_i$ and $b_i$ as maxima of finitely many affine functions and
let $L_{a_i}$ and $L_{b_i}$ be the largest one-norms of their slopes.  The
standard max-affine Lipschitz bound, applied relative to the origin, gives
\[
  |a_i(x)-c_i|\le L_{a_i}\|x\|_\infty,
  \qquad
  |b_i(x)-c_i|\le L_{b_i}\|x\|_\infty,
\]
and consequently
$|\phi_i(x)|\le(L_{a_i}+L_{b_i})\|x\|_\infty$.
Set $K_i=2(L_{a_i}+L_{b_i})$ and define
\begin{align*}
  p_i(x)&=a_i(x)-c_i+K_i\|x\|_\infty,\\
  q_i(x)&=b_i(x)-c_i+K_i\|x\|_\infty.
\end{align*}
These functions are convex, rational, PWA, vanish at the origin, and have
difference $\phi_i$.  Moreover,
\begin{align*}
 p_i(x)&\ge (L_{a_i}+2L_{b_i})\|x\|_\infty
          \ge |\phi_i(x)|,\\
 q_i(x)&\ge (2L_{a_i}+L_{b_i})\|x\|_\infty
          \ge |\phi_i(x)|.
\end{align*}
Finally, $\|x\|_\infty$ and sums of finite rational max-affine functions are
again finite rational max-affine (by enumerating signed coordinates and
pairwise sums).  Hence all operations are effective, although the piece count
may grow.
\end{proof}

For the functions supplied by \cref{lem:epigraph-decomposition}, choose rational
affine representations with finite index sets $\mathcal J_i^p$ and
$\mathcal J_i^q$,
\begin{align*}
 p_i(x)&=\max_{j\in\mathcal J_i^p}
       \{(a_{ij}^p)^{\mathsf T}x+b_{ij}^p\},\\
 q_i(x)&=\max_{j\in\mathcal J_i^q}
       \{(a_{ij}^q)^{\mathsf T}x+b_{ij}^q\}.
\end{align*}

\begin{proposition}[Exact epigraph realization]
\label{prop:fixed-dimensional-realization}
Let $\phi:\R^d\to\R^d$ be a rational continuous PWA map satisfying
$\phi(0)=0$.  Consider the linear plant with state $x_k\in\R^d$ and input
$u_k=(v_k,w_k)\in\R^{2d}$,
\begin{equation}
  x_{k+1}=v_k-w_k
  =\begin{bmatrix}I&-I\end{bmatrix}u_k.
  \label{eq:fixed-dimensional-plant}
\end{equation}
At each state solve the horizon-one MPC problem
\begin{subequations}
\label{eq:fixed-dimensional-mpc}
\begin{align}
 \min_{v_k,w_k,x_{k+1}}\quad&
  \|x_k\|^2+\|v_k\|^2+\|w_k\|^2+\|x_{k+1}\|^2,\\
 \text{subject to}\quad&
  x_{k+1}=v_k-w_k,\\
 &v_{k,i}\ge (a_{ij}^p)^{\mathsf T}x_k+b_{ij}^p,
   \quad j\in\mathcal J_i^p,\\
 &w_{k,i}\ge (a_{ij}^q)^{\mathsf T}x_k+b_{ij}^q,
   \quad j\in\mathcal J_i^q.
\end{align}
\end{subequations}
The last two families of inequalities are imposed for $i=1,\ldots,d$.
This problem is feasible for every state, has a unique optimizer, and generates
the exact closed-loop map $x_{k+1}=\phi(x_k)$.  The origin,
$(x_k,u_k)=(0,0)$, is feasible.
\end{proposition}

\begin{proof}
For fixed $x_k$, the inequalities in \eqref{eq:fixed-dimensional-mpc} are
equivalent to
\[
  v_{k,i}\ge p_i(x_k),\qquad w_{k,i}\ge q_i(x_k).
\]
Thus $(v_k,w_k)=(p(x_k),q(x_k))$ is feasible.  After eliminating the successor
state and omitting the constant current-state term, the objective is
\begin{equation*}
  J(v_k,w_k)=\|v_k\|^2+\|w_k\|^2+\|v_k-w_k\|^2.
\end{equation*}
For brevity, fix $x_k$ and write $p=p(x_k)$ and $q=q(x_k)$.  The feasible set
for $(v_k,w_k)$ is the translated nonnegative orthant
\[
  \mathcal F(x_k)=\{(p+\alpha,q+\beta):\alpha,\beta\in\R_+^d\}.
\]
The gradients of the reduced objective are
\[
  \nabla_v J(v,w)=2(2v-w),
  \qquad
  \nabla_w J(v,w)=2(2w-v).
\]
At the lower epigraph corner $(p,q)$, the corresponding gradient components
satisfy
\begin{align*}
  2p_i-q_i&=p_i+(p_i-q_i)=p_i+\phi_i(x_k)\ge0,\\
  2q_i-p_i&=q_i-(p_i-q_i)=q_i-\phi_i(x_k)\ge0,
\end{align*}
where the inequalities follow from
$p_i(x_k),q_i(x_k)\ge |\phi_i(x_k)|$.  Consequently, for every feasible
$(v,w)=(p+\alpha,q+\beta)$,
\begin{align*}
  \nabla_{(v,w)} J(p,q)^{\mathsf T}
  \begin{bmatrix}v-p\\w-q\end{bmatrix}
  &=2\sum_{i=1}^d(2p_i-q_i)\alpha_i\\
  &\quad+2\sum_{i=1}^d(2q_i-p_i)\beta_i
    \ge0.
\end{align*}
This is the first-order optimality condition for minimizing the differentiable
convex function $J$ over the convex set $\mathcal F(x_k)$, so $(p,q)$ is a
global optimizer.  Moreover, $J$ is strictly convex: its Hessian is
\[
  \nabla^2J=
  \begin{bmatrix}4I&-2I\\-2I&4I\end{bmatrix}\succ0.
\]
Hence the optimizer is unique.  Substitution into the plant gives
\[
  x_{k+1}=p(x_k)-q(x_k)=\phi(x_k).
\]
Finally, $p_i(0)=q_i(0)=0$ implies that every affine piece in each max-affine
representation has value at most zero at the origin.  Hence $(x_k,u_k)=(0,0)$
satisfies every stage inequality.
\end{proof}

\begin{theorem}[Fixed-dimensional MPC undecidability]
\label{thm:fixed-dimensional-mpc}
The global asymptotic-stability decision problem is undecidable already for
rational horizon-one linear MPC with fixed state dimension $n_x=3$, fixed input
dimension $n_u=6$, and weights $Q=I_3$, $R=I_6$, and $P=I_3$.  The quadratic
program has a unique optimizer and is feasible for every state; in particular,
recursive feasibility is automatic.  Only rational mixed state-input
inequalities are required, and their number is part of the instance.
\end{theorem}

\begin{proof}
Given an instance $\phi$ of \cref{thm:fixed-pwa-source}, apply the construction of \cref{lem:epigraph-decomposition} and form
\eqref{eq:fixed-dimensional-mpc} with $d=3$.  By
\cref{prop:fixed-dimensional-realization}, its closed-loop map is exactly
$x_{k+1}=\phi(x_k)$ on $\R^3$.  Thus the MPC closed loop is globally
asymptotically stable if and only if the source PWA system is.  A decision
algorithm for this fixed-dimensional MPC class would therefore contradict
\cref{thm:fixed-pwa-source}.
\end{proof}

\section{Summary and conclusions}

We have shown that deciding global asymptotic stability for constrained linear MPC is undecidable.  Three exact realizations of saturated linear dynamics establish the result separately for predicted-state, input, and quadratically softened input constraints, already with horizon one, fixed weights, unique optimizers, and global feasibility.  A predicted-state variant also gives undecidability on the compact invariant feasible set $[-1,1]^n$.  A fourth, epigraph-based realization of rational continuous PWA dynamics shows that undecidability persists when both dimensions are fixed to $n_x=3$ and $n_u=6$, again with horizon one, identity weights, a unique optimizer, and global feasibility.  Thus, in the subclasses considered here, the obstruction is already present at horizon one and cannot be attributed to any one of the simple constraint types studied above, to growing state or input dimension, to optimizer nonuniqueness, or to a failure of global or recursive feasibility.

This impossibility concerns a universal decision algorithm; it does not preclude stabilizing controller design or the verification of particular controllers through sufficient conditions.  Standard terminal ingredients remain applicable, while SOS/KKT, MILP, invariant-set, and explicit piecewise-affine methods can provide such rigorous certificates \cite{KordaJones2017,SimonLofberg2016}.  There are also decidable subclasses: unconstrained linear-quadratic MPC reduces to an eigenvalue test, while stability also can follow by construction under standard terminal-set, terminal-controller, and terminal-cost decrease conditions.  \emph{Outside such subclasses, a general method must permit an inconclusive answer, restrict itself to sufficient conditions, or fail to terminate on some instances}.

Finally, the fixed-dimensional construction uses six inputs and general mixed state-input inequalities.  Whether undecidability persists for a few-state plant with scalar input, or under a fixed simple constraint structure remains open.

\section*{Acknowledgement}
The author used OpenAI GPT 5.6 Sol for improving constructions in Sections III--IV. OpenAI Codex GPT 5.6 Sol assisted with partially drafting and proof-reading the manuscript. The author independently verified all AI-generated content and takes full responsibility for the article.

\bibliographystyle{IEEEtran}
\bibliography{references}

\end{document}